\documentclass[11pt]{amsart}

\usepackage[margin=1.2in]{geometry}
\usepackage{hyperref}
\usepackage{amsfonts}
\usepackage{amsmath}
\usepackage{amsthm}
\usepackage{amssymb}
\usepackage{float}
\usepackage{enumerate}
\usepackage{graphicx}
\usepackage{mathrsfs}
\usepackage{xcolor}
\usepackage[all]{xy}
\usepackage{bm}
\usepackage[percent]{overpic}

\newtheorem{theor}{Theorem}[section]
\newtheorem{cor}[theor]{Corollary}
\newtheorem{lemma}[theor]{Lemma}
\newtheorem{prop}[theor]{Proposition}

\newtheorem{theorintro}{Theorem}

\theoremstyle{definition}
\newtheorem{rmk}[theor]{Remark}

\DeclareMathOperator{\cc}{\mathcal{C}}
\DeclareMathOperator{\pg}{\mathcal{P}}
\newcommand{\mg}[2]{\mathcal{C}^{[#1]}(#2)}
\DeclareMathOperator{\nc}{\mathcal{NC}}
\DeclareMathOperator{\aut}{Aut}
\DeclareMathOperator{\mcg}{Mod}
\DeclareMathOperator{\im}{Im}
\DeclareMathOperator{\St}{St}

\DeclareMathOperator{\bc}{\mathcal{B}}
\DeclareMathOperator{\ac}{\mathcal{A}}
\newcommand{\bg}[2]{\mathcal{B}^{[#1]}(#2)}
\newcommand{\pos}[1]{\left({#1}\right)^{pos}}

\title{Simplicial embeddings between multicurve graphs}
\author{Viveka Erlandsson}
\address[Viveka Erlandsson]{Aalto Science Institute, Aalto University, Finland \& University of Fribourg, Switzerland}
\email{viveca.erlandsson@aalto.fi}

\author{Federica Fanoni}
\address[Federica Fanoni]{Mathematics Institute, University of Warwick, UK}
\email{F.Fanoni@warwick.ac.uk}

\date{\today}
\date{\today}

\begin{document}
\begin{abstract}
We study some graphs associated to a surface, called $k$-multicurve graphs, which interpolate between the curve complex and the pants graph. Our main result is that, under certain conditions, simplicial embeddings between multicurve graphs are induced by $\pi_1$-injective embeddings of the corresponding surfaces. We also prove the rigidity of the multicurve graphs.
\end{abstract}
\maketitle

\section{Introduction}
Since the introduction of graphs associated to surfaces, there have been different questions raised about them. One line of questions concerns maps between such graphs. In particular, as the mapping class group usually acts on these complexes by automorphisms, it is interesting to understand whether these are all automorphisms or there are some not induced by a self-homeomorphism of the underlying surface. This is the {\it rigidity} problem: a graph associated to a surface is {\it rigid} if its automorphism group is the mapping class group of its surface. The first proof of the rigidity of the curve graph is due to Ivanov \cite{ivanov}, in the case of surfaces of genus at least two; subsequently Korkmaz \cite{korkmaz} proved it for low genus surfaces and a proof in the general is due to Luo \cite{luo}. In \cite{margalit}, Margalit proved the rigidity of the pants graph by reducing it to the rigidity of the curve complex. Multiple other graphs have been proven to be rigid (like the Hatcher-Thurston graph \cite{irmak} and the graph of nonseparating curves \cite{ik}), and it is conjectured that graphs associated to surfaces not having obvious obstructions (such as being disconnected) should be rigid.

A more subtle problem is studying {\it simplicial embeddings} (i.e.\ injective maps preserving the graph structure) between such objects. There are embeddings which correspond to maps between the underlying surfaces: for instance, in the case of the curve graph, a $\pi_1$-injective embedding of a surface into another (that is, seeing the first surface as an essential subsurface of the second) induces a simplicial embedding between the curve graphs. A similar construction holds for the pants graph: suppose $S_1$ is a subsurface of $S_2$ and choose of a $k$-multicurve $\nu$ in the complement of the subsurface, where $k$ is the difference of the complexities of the surfaces. Then we get a simplicial embedding of the pants graph of $S_1$ into the pants graph of $S_2$ by sending a pants decomposition $\mu$ of $S_1$ to $\mu\cup \nu$. In \cite{aramayona}, Aramayona showed that, except in some low complexity cases, any simplicial embedding between pants graphs arises this way. On the other hand, for curve graphs there are examples of simplicial embeddings not coming from $\pi_1$-injective embeddings between the underlying surfaces. For instance, puncturing a surface gives a simplicial embedding between the curve graph of the surface and the curve graph of the punctured one (see \cite{as} for a precise description of this). Note also that in \cite{aramayona2}, Aramayona described conditions for graphs built from arcs and/or curves to satisfy a similar superrigidity property. It is natural to ask which graphs, besides the pants graph, satisfy these conditions.

Among the reasons to study maps between graphs associated to surfaces one is the possibility of using such results to say more about mapping class groups. Examples of these results are Ivanov and McCarthy's results in \cite{ivanov2} and \cite{mccarthy}, saying that for surfaces of genus at least three, the group of automorphisms of the mapping class group is the extended mapping class group. A tool in their proof is Irmak's rigidity of the complex of nonseparating curves. In a similar way, rigidity of the curve complex and of other complexes built from multicurves is used by Ivanov in \cite{ivanov} to show that automorphisms between finite index subgroups of the mapping class groups are restrictions of automorphisms of the whole mapping class group.

In this paper, we are interested in a new set of graphs, called {\it $k$-multicurve graphs} (for $k$ between one and the complexity of the surface). These graphs can be thought of as graphs interpolating between the curve graph and the pants graph. Indeed, vertices of the $k$-multicurve graph $\mg{k}{S}$ are $k$-multicurves and edges correspond to minimal intersection, in such a way that for $k=1$ we obtain the curve graph and for $k$ equal to the complexity the pants graph.

Note that other graphs interpolating between the curve and the pants graph have been constructed by Mj in \cite{mj}, where he computes their geometric rank.

Just as in the case of pants graphs, if $S_1$ can be seen as a subsurface of $S_2$, we can see any multicurve graph of $S_1$ as subgraph of a multicurve graph of $S_2$. Indeed, fixing a $d$-multicurve $\nu$ on $S_2$ disjoint from $S_1$ and mapping a $k$-multicurve $\mu$ of $S_1$ to $\mu\cup\nu$ gives a simplicial embedding of $\mg{k}{S_1}$ into $\mg{k+d}{S_2}$. In general though, not every simplicial embedding between multicurve graphs is induced by a map between the underlying surfaces, as examples in Section \ref{ex} show. On the other hand, if we give some condition on the topology, we can control exactly how simplicial embedding arise. This is the content of our main theorem:
\begin{theorintro}\label{embedding_intro}
Let $S_1$ and $S_2$ be connected finite type surfaces such that the complexity $\xi(S_1)$ is at least $4+k_1$. Let $\varphi:\mg{k_1}{S_1}\hookrightarrow \mg{k_2}{S_2}$ be a simplicial embedding, with $k_2\geq k_1$, and assume $\xi(S_2)-\xi(S_1)\leq k_2-k_1$. Then 
$$\xi(S_2)-\xi(S_1)=k_2-k_1$$
and:
\begin{itemize}
\item if $k_2=k_1$, $\varphi$ is an isomorphism induced by a homeomorphism $f:S_1\rightarrow S_2$;
\item if $k_2>k_1$, there exist a $\pi_1$-injective embedding $f:S_1\rightarrow S_2$ and a $(k_2-k_1)$-multicurve $\nu$ on $S_2$ such that for any $\mu\in\mg{k_1}{S_1}$ we have
$$\varphi(\mu)=f(\mu)\cup\nu.$$
\end{itemize}
\end{theorintro}

An analogous control on the topology is necessary when trying to describe injective homomorphisms between mapping class groups. While in general one cannot say much, restricting the topology allows us to say that injective maps between (extended) mapping class groups are induced by maps between the surfaces (see the results of Aramayona and Souto \cite{as2}, Ivanov and McCarthy \cite{im} and Bell and Margalit \cite{bm}).

While proving Theorem \ref{embedding_intro} we also show that the multicurve graphs are rigid.
\begin{theorintro}\label{rigidity_intro}
Let $S$ be a finite type surface of complexity $\xi(S)$ at least four. For any $k$ between $1$ and $\xi(S)$, the natural map
$$\mcg(S)\rightarrow\aut\left(\mg{k}{S}\right)$$
is a group isomorphism.
\end{theorintro}

This paper is organized as follows. In Section \ref{preliminaries} we introduce the objects of study and we prove some basic results. Theorem \ref{embedding_intro} is proved by induction on $k_1$ and Section \ref{base} contains the proof of the base case. Using results of Section \ref{base}, in Section \ref{rigidity} we deduce the rigidity of the multicurve graphs (Theorem \ref{rigidity_intro}). In Section \ref{induction} we complete the proof of our main theorem by showing the induction step.

\subsection*{Acknowledgements}
This work started when both authors were visiting the CUNY Graduate Center. The authors would like to thank CUNY and in particular Ara Basmajian for the hospitality. They are grateful to Javier Aramayona, Brian Bowditch, Hugo Parlier and Juan Souto for useful discussions. They would especially like to thank Javier Aramayona for sharing the manuscript of \cite{aramayona2}.

Both authors acknowledge support from Swiss National Science Foundation grant number PP00P2\_128557 and the second author also acknowledges support from Swiss National Science Foundation grant number P2FRP2\_161723 and U.S. National Science Foundation grants DMS 1107452, 1107263, 1107367 ``RNMS: GEometric structures And Representation varieties'' (the GEAR Network).
\section{Preliminaries}\label{preliminaries}
Throughout the paper, we will assume all surfaces to be of finite type, connected and of negative Euler characteristic.

\subsection{Definitions and notation.}\label{defs} 
(Almost) all graphs we consider are defined using curves on surfaces. Here by {\it curve} we mean the free homotopy class of a simple closed curve, not homotopic to a point or a puncture. We say that two curves $\alpha$ and $\beta$ are {\it disjoint} if there are representatives $a$ of $\alpha$ and $b$ of $\beta$ that are disjoint.

A {\it $k$-multicurve} is a collection of $k$ disjoint curves. The maximal size of a multicurve is the {\it complexity} $\xi=\xi(S)=3g(S)-3+n(S)$ of $S$; here $g(S)$ and $n(S)$ are the genus and the number of punctures of $S$, respectively. A maximal size multicurve is a {\it pants decomposition}. Given a curve $\alpha$ or a multicurve $\mu$, $S\setminus \alpha$ or $S\setminus \mu$ denotes the surface obtained by removing, respectively, a representative of $\alpha$ or disjoint representatives of the curves of $\mu$. Note that for us cutting along a curve gives rise to two punctures (and not two boundary components, as for other authors). From the point of view of the multicurve graphs, there is no difference between surfaces with boundary or punctures and considering open surfaces makes some of our proofs cleaner.

We say that a curve $\alpha$ is {\it nonseparating} if $S\setminus \alpha$ is connected. The curve is an {\it outer} curve if $S\setminus\alpha$ has a component of complexity zero (i.e.\ a {\it pair of pants}). Given two disjoint outer or nonseparating curves $\alpha$ and $\beta$, we say that they form a {\it nice pair} if $S\setminus (\alpha\cup\beta)$ has exactly one positive complexity component. Note that as a consequence of the definition, if $(\alpha,\beta)$ is a nice pair, then the unique component of $S\setminus (\alpha\cup\beta)$ which is not a pair of pants has complexity $\xi(S)-2$.

\begin{figure}[h]
\begin{overpic}{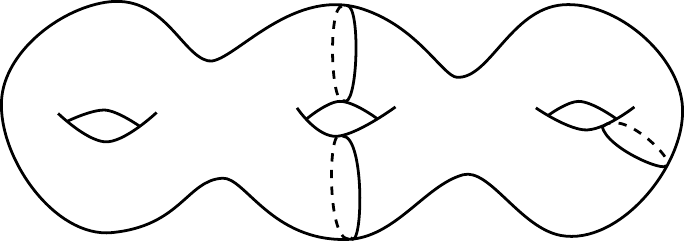}
\put (55,6) {$\alpha$}
\put (98,6) {$\beta$}
\put (53,27) {$\gamma$}
\end{overpic}
\caption{$\alpha$ and $\beta$ form a nice pair, while $\alpha$ and $\gamma$ don't.}
\end{figure}

We study the {\it $k$-multicurve graphs}, for $k\in\{1,\dots,\xi(S)\}$. For a given $k$, the vertices of $\mg{k}{S}$ are $k$-multicurves. Two multicurves are joined by an edge if they have a common $(k-1)$-multicurve $\nu$ and the remaining two curves intersect minimally on $S\setminus \nu$. Note that the minimum number of intersections of two curves on $S\setminus \nu$ is
\begin{itemize}
\item zero, if $\xi(S\setminus \nu)\geq 2$,
\item one, if $S\setminus \nu$ is a once-punctured torus,
\item two, if $S\setminus \nu$ if a four-punctured sphere.
\end{itemize}
With this definition, $\mg{1}{S}$ it the curve graph and $\mg{\xi(S)}{S}$ is the pants graph, so the multicurve graphs can be seen as graphs interpolating between $\cc(S)$ and $\pg(S)$.

We also need some definitions from graph theory. Given a graph $G$, we denote by $V(G)$ its set of vertices and by $E(G)$ its set of edges. A subgraph $H$ of $G$ is an {\it induced} subgraph if any two vertices of $H$ which are adjacent in $G$ are adjacent in $H$ as well. If $W$ is a set of vertices of $G$, the subgraph {\it spanned} by $W$ is the induced subgraph $H$ with $V(H)=W$. If $v$ is a vertex, the {\it star} of $v$, denoted $\St(v)$, is the subgraph spanned by $v$ and all its adjacent vertices.

An important induced subgraph of $\mg{k}{S}$ is the graph spanned by all multicurves containing $\nu$, where $\nu$ is a multicurve (possibly reduces to a single curve). We denote this graph by $C_\nu$. If $d=|\nu|$, there is a natural isomorphism $C_\nu\simeq\mg{k-d}{S\setminus \nu}$, given by $\mu\mapsto\mu\setminus \nu$.
\begin{rmk}\label{int_complete}
If $\nu_1$ and $\nu_2$ are two distinct $(k-1)$-multicurves such that {$C_{\nu_1}\cap C_{\nu_2}\neq \emptyset$}, then $C_{\nu_1}\cap C_{\nu_2}$ is exactly one vertex, which is the multicurve $\nu_1\cup\nu_2$. In particular, $|\nu_1\cap\nu_2|=k-2$ (that is, the intersection has maximal possible size).
\end{rmk}
\subsection{Connectedness results}\label{connectedness}
The connectivity of the multicurve graphs can be shown with the same ideas used in \cite[Section 1.4]{mm2} to prove that two pants decompositions of the five-holed sphere can be connected by a path in the pants graph. We reproduce the argument for completeness.
\begin{lemma}\label{connected}
The multicurve graphs are connected.
\end{lemma}
\begin{proof}
We proceed by induction on $k$. The base case, for $k=1$, is the connectivity of the curve graph, which is well known.

Suppose now $\mg{k-1}{\Sigma}$ is connected for every surface satisfying $\xi(\Sigma)\geq k-1$. Consider a surface $S$ of complexity at least $k$ and two $k$-multicurves $\nu=\{\alpha_1,\dots,\alpha_k\}$ and $\nu'=\{\beta_1,\dots,\beta_k\}$. Since the curve graph is connected, there is a path in $\cc(S)$ between $\alpha_1$ and $\beta_2$, say $\alpha_1=\gamma_0,\dots,\gamma_m=\beta_1$. Choose a $k$-multicurve $\nu_1$ which contains $\alpha_1$ and $\gamma_1$. Then by induction there is a path in $\mg{k-1}{S\setminus\alpha_1}$ between $\nu\setminus\{\alpha_1\}$ and $\nu_1\setminus \{\alpha_1\}$, hence a path in $\mg{k}{S}$ between $\nu$ and $\nu_1$. In the same way, we construct multicurves $\nu_2,\dots\nu_{m}$ such that $\nu_i$ contains $\gamma_{i-1}$ and $\gamma_i$, for $i\leq m$, and paths between $\nu_i$ and $\nu_{i+1}$, for $i\leq m$, and between $\nu_m$ and $\nu'$. The concatenation of these gives a path in $\mg{k}{S}$ between $\nu$ and $\nu'$.
\end{proof}

For the proof of Theorem \ref{embedding_intro} we need to consider the subgraph $\bg{k}{S}$ of $\mg{k}{S}$ given by:
\begin{itemize}
\item $V\left(\bg{k}{S}\right)=\{\mu\, | \,\mu \, \mbox{contains a nonseparating or outer curve}\},$
\item $E\left(\bg{k}{S}\right)=\{\mu\nu\, | \,|\mu\cap\nu|=k-1 \, \mbox{and the remaining curves form a nice pair}\}.$
\end{itemize}
For $k=1$ we simply denote the graph by $\bc(S)$. For $k=\xi(S)$ there is no edge in $\bg{k}{S}$, but its set of vertices is the same as the set of vertices of $\mg{k}{S}=\pg(S)$. In general the graph is not an induced subgraph of the multicurve graph, since there are pairs of disjoint curves which do not form nice pairs.

To prove that $\bg{k}{S}$ is connected, for $\xi(S)$ big enough, we consider two auxiliary graphs. For a surface with genus, we denote by $\nc(S)$ the subgraph of $\cc(S)$ spanned by nonseparating curves. 

If $S$ has punctures, by {\it arc} on $S$ we mean the homotopy class of an embedded arc, starting and ending at punctures, which is not nullhomotopic (relative to the punctures). The {\it arc graph} $\ac(S)$ is the graph whose vertices are arcs and whose edges correspond to pair of arcs that can be realized disjointly.

\begin{lemma}\label{bconn}
If $\xi(S)\geq 3$, the graph $\bc(S)$ is connected.
\end{lemma}
\begin{proof}
Let $g$ be the genus of $S$.

If $g=0$, we have a bijection between outer curves and the subset of $V\left(\ac(S)\right)$ given by arcs with two different endpoints: given an arc between two different punctures, the associated outer curve is the boundary of a regular neighborhood of the arc.

Consider two outer curves $\alpha$ and $\beta$ on $S$. Because of the condition on the complexity of $S$, they form a nice pair if and only if they are disjoint. So in this case $\bc(S)$ is an induced subgraph of $\cc(S)$ (spanned by outer curves).

Suppose $\alpha$ and $\beta$ intersect and consider the associated arcs $\widetilde{\alpha}$ and $\widetilde{\beta}$. If $\widetilde{\alpha}$ and $\widetilde{\beta}$ are adjacent in $\ac(S)$, they must share at least one endpoint and it is easy to find an arc disjoint from both sharing no endpoints with either of them (see Figure \ref{commonendpts}). So we get a path between $\alpha$ and $\beta$.
\begin{figure}[h]
\begin{overpic}{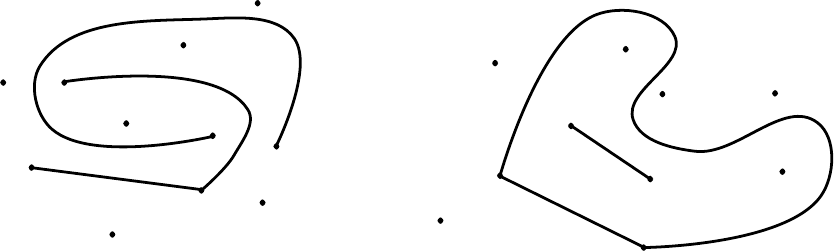}
\put (12,4) {$\widetilde{\alpha}$}
\put (17,22) {$\widetilde{\beta}$}
\put (37,20) {$\widetilde{\gamma}$}
\put (23,-6) {(1)}
\put (65,1) {$\widetilde{\alpha}$}
\put (85,14) {$\widetilde{\beta}$}
\put (72,7) {$\widetilde{\gamma}$}
\put (75,-6) {(2)}
\end{overpic}
\vspace*{0.5cm}
\caption{Arcs sharing one or two endpoints and the corresponding arc disjoint from both. In case (2), there are enough cusps either inside or outside the curve formed by $\widetilde{\alpha}$ and $\widetilde{\beta}$.}\label{commonendpts}
\end{figure}
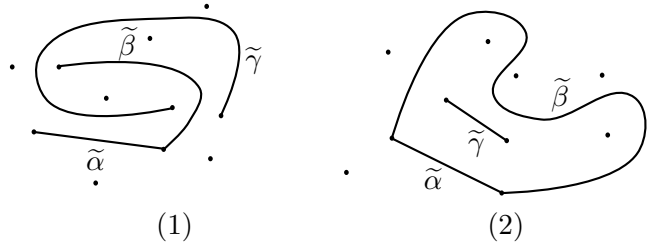
Suppose then $\widetilde{\alpha}$ and $\widetilde{\beta}$ are not adjacent in $\ac(S)$. Using the construction of {\it unicorn paths} described in \cite{hpw}, we obtain a path $\widetilde{\gamma_0}=\widetilde{\alpha},\dots,\widetilde{\gamma_m}=\widetilde{\beta}$ in $\ac(S)$ such that no $\widetilde{\gamma_j}$ starts and ends at the same puncture: indeed, a unicorn path from $\widetilde{\alpha}$ to $\widetilde{\beta}$ contains arcs starting at an endpoint of $\widetilde{\alpha}$ and ending at an endpoint of $\widetilde{\beta}$, and these can be chosen to be different.
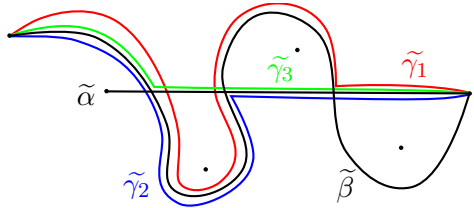
\begin{figure}[h]
\begin{overpic}{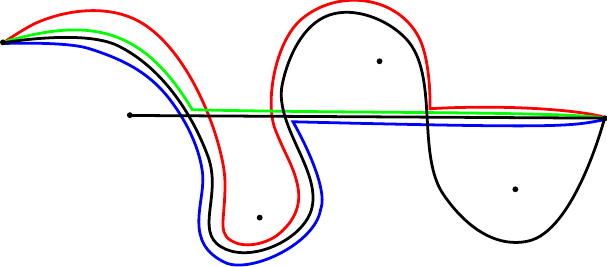}
\put (15,22) {$\widetilde{\alpha}$}
\put (71,2) {$\widetilde{\beta}$}
\put (85,29) {\textcolor{red}{$\widetilde{\gamma_1}$}}
\put (25,3) {\textcolor{blue}{$\widetilde{\gamma_2}$}}
\put (56,28) {\textcolor{green}{$\widetilde{\gamma_3}$}}
\end{overpic}
\caption{A unicorn path between the arcs $\widetilde{\alpha}$ and $\widetilde{\beta}$}\label{unicorn}
\end{figure}

Consider the corresponding outer curves $\gamma_0=\alpha,\dots,\gamma_m=\beta$. For any $j$, the arcs $\widetilde{\gamma_j}$ and $\widetilde{\gamma_{j+1}}$ share two endpoints. As described before, there is an outer curve $\delta_j$ which is disjoint from both $\gamma_j$ and $\gamma_{j+1}$. Then
$$\alpha=\gamma_0,\delta_0,\gamma_1,\delta_1,\dots,\gamma_m=\beta$$
is a path in $\bc(S)$.

If $g\geq 1$, for any outer curve there is a disjoint nonseparating one (and vice versa). Moreover, if an outer curve and a nonseparating one are disjoint, they form a nice pair. So to show that $\bc(S)$ is connected, it is enough to prove that there is a path in $\bc(S)$ between any two nonseparating curves. Let $\alpha$ and $\beta$ be two such curves.

Suppose $g=1$. As proven by Schmutz Schaller in \cite{schmutz}, there is a sequence $\gamma_0=\alpha,\dots, \gamma_m=\beta$ of nonseparating curves such that $\gamma_j$ and $\gamma_{j+1}$ intersect exactly once for every $j$. Cutting $S$ along $\gamma_j$ and $\gamma_{j+1}$  gives us a punctured disk (with at least three punctures) and it is clear that there is an outer curve $\delta_j$ disjoint from $\gamma_j$ and $\gamma_{j+1}$ (see Figure \ref{torus}).
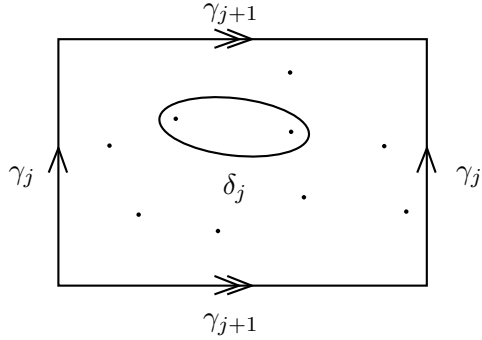
\begin{figure}[h]
\begin{overpic}{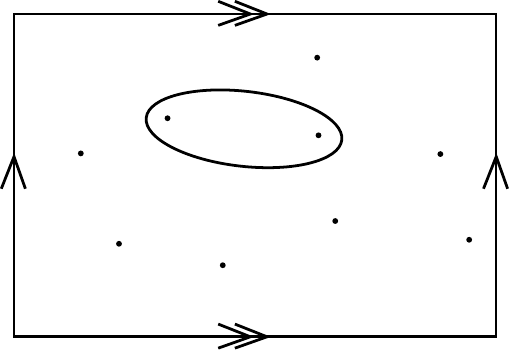}
\put (-10,30) {$\gamma_j$}
\put (105,30) {$\gamma_j$}
\put (45,26) {$\delta_j$}
\put (40,-8) {$\gamma_{j+1}$}
\put (40,72) {$\gamma_{j+1}$}
\end{overpic}
\caption{Cutting along $\gamma_j$ and $\gamma_{j+1}$.}\label{torus}
\end{figure}
So the sequence
$$\alpha=\gamma_0,\delta_0,\gamma_1,\delta_1,\gamma_2,\dots,\delta_{m-1},\gamma_m=\beta$$
is a path in $\bc(S)$.

Suppose now $g\geq 2$. By \cite[Theorem 4.4]{fm}, the graph $\nc(S)$ is connected. Let $\gamma_0=\alpha,\dots,\gamma_m=\beta$ be a path in $\nc(S)$. If for some $j$ the curves $\gamma_j$ and $\gamma_{j+1}$ do not form a nice pair, then $S\setminus(\gamma_j\cup\gamma_{j+1})$ must have two connected components, $\Sigma_1$ and $\Sigma_2$. As the genus of $S$ is at least two, one component has genus and we can pick a nonseparating curve $\delta_j$ on it. Then $(\gamma_j,\delta_j)$ and $(\delta_j,\gamma_{j+1})$ are nice pairs. By adding the curves $\delta_j$ when needed we get a path in $\bc(S)$ between $\alpha$ and $\beta$.
\end{proof}

\begin{cor}
If $\xi(S)\geq k+2$, the graph $\bg{k}{S}$ is connected.
\end{cor}
\begin{proof} The proof is by induction on $k$. For $k=1$ it is the result of Lemma \ref{bconn}. Suppose $k\geq 2$ and let $\mu,\nu\in\bg{k}{S}$. Fix $\alpha\in \mu$ and $\beta\in \nu$ such that both $\mu\setminus \{\alpha\}$ and $\nu\setminus \{\beta\}$ still contain a nonseparating or outer curve. By the connectivity of $\bc(S)$, there is a path $\gamma_0=\alpha,\gamma_1,\dots,\gamma_m=\beta$ in $\bc(S)$. Choose a $k$-multicurve $\eta_1$ containing $\alpha$ and $\gamma_1$. Because $\bg{k-1}{S\setminus \alpha}$ is connected by the induction hypothesis, there is a path in it between $\mu\setminus \alpha$ and $\eta_1\setminus \alpha$, which gives a path between $\mu$ and $\eta_1$ in $\bg{k}{S}$. By repeating the same argument we obtain a path in $\bg{k}{S}$ between $\mu$ and $\nu$.
\end{proof}

\begin{rmk}\label{adjacent}
Any multicurve $\mu$ in $\mg{k}{S}\setminus\bg{k}{S}$ is adjacent to infinitely many multicurves $\nu$ in $\bg{k}{S}$: if $k=\xi(S)$, it is vacuously true because $V\left(\mg{k}{S}\right)=V\left(\bg{k}{S}\right)$. Otherwise:
\begin{itemize}
\item if $S$ has genus there is a nonseparating curve $\alpha$ disjoint from $\mu$ and we set $\nu:=\mu\cup\{\alpha\}\setminus \{\beta\}$ (for some $\beta\in \mu$);
\item if $S$ is a sphere, as $\mu$ does not contain any outer curve, there is a component of $S\setminus \mu$ which is a four-holed sphere with only one boundary component belonging to $\mu$. Choose any outer curve $\alpha$ in the four-holed sphere and set $\nu:=\mu\cup\{\alpha\}\setminus \{\beta\}$ (for some $\beta\in \mu$).
\end{itemize}
As in both cases there are infinitely many choices for the curve $\alpha$, we get the claim.
\end{rmk}

\subsection{(Counter)examples}\label{ex}
As mentioned in the introduction, Aramayona stated in \cite{aramayona2} some conditions ensuring that graphs built from arcs and/or curves satisfy the following superrigidity property: any {\it alternating} simplicial embedding between such graphs is induced by a map at the surface level. We refer to Aramayona's paper for the precise statement and the necessary definitions. Here we show that the multicurve graphs do not satisfy these conditions. Suppose $k\geq 2$ and consider a $(k-2)$-multicurve $\nu$ on a surface $S$ of complexity at least $k+1$. Let $\alpha$, $\beta$ and $\gamma$ be curves on $S$, mutually disjoint and disjoint from $\nu$. Then $\mu_1=\{\alpha,\beta\}\cup \nu$, $\mu_2=\{\beta,\gamma\}\cup \nu$ and $\mu_3=\{\alpha,\gamma\}\cup \nu$ form a triangle and their intersection is $\nu$, which is an extendable set of deficiency $2$. Hence, for the conditions to be satisfied, there should be an alternating circuit containing them. However, as $\mu_i\cap\mu_j$ contains $k-1$ curves for all pairs $i\neq j$, they cannot be contained in any alternating circuit.

We give now two examples showing that the hypothesis $\xi(S_2)-\xi(S_1)\leq k_2-k_1$ in Theorem \ref{embedding_intro} is necessary. Both examples are simple generalizations of similar constructions for the curve complex (see \cite{aramayona} and \cite{as}).

{\bf First example, for $\bm{k_1=k_2}$:} consider a closed surface $S$ of genus $g\geq 2$ and choose $k\leq 3g-4$. Fix a hyperbolic structure on $S$. It is known (see \cite{bs}) that the complement on $S$ of the union of all simple closed geodesics is nowhere dense. We can thus pick $n$ points $p_1,\dots,p_n$ in the complement of all simple closed curves and we have a natural way to associate to a curve on $S$ a curve on the surface $\widetilde{S}:=S\setminus \{p_1,\dots,p_n\}$: for $\alpha$ on $S$, its geodesic representative lies in $\widetilde{S}$, so we can consider the homotopy class of the geodesic representative in $\widetilde{S}$. Because this map respects disjointness, it gives a simplicial embedding $\mg{k}{S}\hookrightarrow \mg{k}{\widetilde{S}}$. $S$ is closed, so the embedding cannot be induced by an embedding of the surfaces.

{\bf Second example, for $\bm{k_2>k_1}$:} consider a closed surface $S$  of complexity at least $k+1$ (for some positive $k$) and fix a hyperbolic structure on it. Choose $d$ points $p_1,\dots,p_d$ in the complement of the union of all simple closed geodesics on $S$ to obtain an embedding $\mg{k}{S}\hookrightarrow \mg{k}{S\setminus\{p_1,\dots,p_d\}}$ as before. Clearly
$$\mg{k}{S\setminus\{p_1,\dots,p_d\}}\simeq \mg{k}{\widetilde{S}},$$
where $\widetilde{S}$ has the same genus as $S$ and $d$ boundary components, so we have a simplicial embedding $\Phi:\mg{k}{S}\rightarrow\mg{k}{\widetilde{S}}$. By gluing pairs of pants to each boundary component of $\widetilde{S}$ we get a new surface $\Sigma$ of the same genus as $S$, with $2d$ boundary components and a distinguished $d$-multicurve $\nu$ (given by the curves corresponding to the boundary components of $\widetilde{S}$). Then
\begin{align*}
\varphi:\mg{k}{S}&\rightarrow\mg{k+d}{\Sigma}\\
\mu&\mapsto \Phi(\mu)\cup \nu
\end{align*}
is a simplicial embedding which cannot be induced by an embedding at the surface level.

\section{Embeddings of the curve graph into multicurve graphs}\label{base}
In this section we will prove Theorem \ref{embedding_intro} for $k_1=1$. We actually need a slightly weaker assumption on the complexity of $S_1$. We prove the following.
\begin{prop}
Let $S_1$ and $S_2$ be surfaces such that $\xi(S_1)\geq 4$. Consider a simplicial embedding 
$$\varphi:\cc(S_1)\hookrightarrow\mg{k}{S_2},$$
where $\xi(S_2)-\xi(S_1)\leq k-1$ and $k\geq 1$. Then $\xi(S_2)-\xi(S_1)=k-1$ and:
\begin{itemize}
\item if $k=1$, $\varphi$ is an isomorphism induced by a homeomorphism $f:S_1\rightarrow S_2$;
\item if $k>1$, then there exist a $\pi_1$-injective embedding $f:S_1\rightarrow S_2$ and a $(k-1)$-multicurve $\nu$ on $S_2$ such that for any $\mu\in\mg{k_1}{S_1}$ we have
$$\varphi(\mu)=f(\mu)\cup\nu.$$
\end{itemize}
\end{prop}

\begin{proof}
If $k=1$, $\xi(S_2)\leq \xi(S_1)$, so by Shackleton's result in \cite{shackleton} $\xi(S_2)=\xi(S_1)$ and $\varphi$ is an isomorphism induced by a homeomorphism $f:S_1\rightarrow S_2$.

Suppose $k\geq 2$. Let $\alpha\in\cc(S_1)$ and consider its image $a=\varphi(\alpha)=\{\alpha_1,\dots,\alpha_k\}$. Consider $\beta$ adjacent to $\alpha$; then $b=\varphi(\beta)$ is adjacent to $a$, so we can assume, without loss of generality, that $b=\varphi(\beta)=\{\alpha_1,\dots,\alpha_{k-1},\beta_k\}$.

Note that $\St(a)\subseteq C_{\alpha_1}\cup C_{\nu_1}$, for $\nu_1=\{\alpha_2,\dots,\alpha_k\}$.
\begin{rmk}\label{Stardec}If $k<\xi(S_2)$, given $x\in\St(a)\cap C_{\alpha_1}$ different from $a$, there exists a unique $x'\in\St(a)\cap C_{\nu_1}$ different from $a$ which is adjacent to $x$ (if $x=\{\alpha_1,\dots,\beta_i,\dots,\alpha_k\}$, then $x'=\nu_1\cup\{\beta_i\}$). Conversely, given $x'\in\St(a)\cap C_{\nu_1}$  different from $a$ there exists a unique $x\in\St(a)\cap C_{\alpha_1}$ which is adjacent to $x'$ and is different from $a$. If $k=\xi(S_2)$, given $x\in\St(a)\cap C_{\alpha_1}$ different from $a$, there exists no $x'\in\St(a)\cap C_{\nu_1}$ different from $a$ which is adjacent to $x$.
\end{rmk}

We want to show that $\varphi(\St(\alpha))\subseteq C_{\alpha_1}$. Let $\gamma\in \St(\alpha)$, different from $\beta$. Since $\St(\alpha)\setminus\{\alpha\}\simeq \cc(S_1\setminus \alpha)$, which is connected, there exists a path
$$\gamma_0=\gamma,\gamma_1,\dots,\gamma_m=\beta$$
in $\St(\alpha)\setminus\{\alpha\}$. Consider $\gamma_{m-1}$ and its image $c_{m-1}=\varphi(\gamma_{m-1})$. As $c_{m-1}$ is adjacent to $b$, by Remark \ref{Stardec} it must belong to $C_{\alpha_1}$ if $k=\xi(S_2)$. If $\xi(S_2)>k$ and we assume by contradiction $c_{m-1}\in C_{\nu_1}$, again by Remark \ref{Stardec} $c_{m-1}$ is completely determined by being adjacent to $b$. Since $\xi(S_1)\geq 4$, there exists a curve $\delta$ on $S_1$ which is disjoint from $\alpha$, $\beta$ and $\gamma_{m-1}$. So $d=\varphi(\delta)\in\varphi(\St(\alpha))$ and is adjacent both to $b$ and $c_{m-1}$. This implies that either $d=b$ or $d=c_{m-1}$, which contradicts the injectivity of $\varphi$. Thus again $c_{m-1}\in C_{\alpha_1}$. By repeating the argument, we deduce that $\gamma\in C_{\alpha_1}$, so $\varphi(\St(\alpha))\subseteq C_{\alpha_1}$.

Note that the proof works for every curve in the intersection $a\cap b$, so what we actually obtain is that $\varphi(\St(\alpha))\subseteq C_\nu$, where $\nu=a\cap b$.

Now, consider any other curve $\eta$ on $S_1$. We can construct a path between $\alpha$ and $\eta$ and by what we just proved the image of any edge in the path is in $C_\nu$. So $\im(\varphi)\subseteq C_{\nu}$. We have\\

\centerline{
\xymatrix{
\cc(S_1)\ar@{^{(}->}[r]^{\varphi}\ar[rd]_\Phi&\ar[d]_\simeq^\theta C_{\nu}\\
&  \cc(S_2\setminus \nu)
}
}

where $\theta$ is the isomorphism given by $\theta(\nu\cup\{\varepsilon\})=\varepsilon$ and $\Phi$ is defined to be the composition $\theta\circ \varphi$. Since $\varphi$ is a simplicial embedding, $\Phi$ is a simplicial embedding too. Moreover $\xi(S_2\setminus\nu)\leq\xi(S_1)$, so by Shackleton's result in \cite{shackleton} $\xi(S_2\setminus \nu)=\xi(S_1)$, i.e.\ $\xi(S_2)-\xi(S_1)=k-1$, and $\Phi$ is induced by a homeomorphism $f:S_1\rightarrow S_2\setminus\nu$. This gives us a $\pi_1$-injective map $F:S_1\rightarrow S_2$ (given by composing $f$ with the natural injection $S_2\setminus \nu\hookrightarrow S_2$) and $\varphi$ is induced by $F$ and $\nu$.
\end{proof}

\section{Rigidity of multicurves graphs}\label{rigidity}
This section is dedicated to the proof of Theorem \ref{rigidity_intro}, which is done by induction on $k$. Recall that the mapping class group $\mcg(S)$ of a surface $S$ is the group of homeomorphisms of $S$ up to isotopy. The action
\begin{align*}
\mcg(S)\times \mg{k}{S}&\rightarrow\mg{k}{S}\\
(f,\alpha)&\mapsto f(\alpha):=\mbox{homotopy class of $f_0(\alpha_0)$, where $f_0\in f$ and $\alpha_0\in \alpha$}
\end{align*}
is well defined and induces a map $\mcg(S)\rightarrow\aut\left(\mg{k}{S}\right)$. It is easy to see that this map is a group homomorphism and the main content of Theorem \ref{rigidity_intro} is that, for $\xi(S)\geq 4$, it is a bijection.

For $k=1$, $\mg{1}{S}$ is the $1$-skeleton of the curve complex, which is rigid if $\xi(S)\geq 4$ (see \cite{ivanov}, \cite{korkmaz}and \cite{luo}).

For the induction step, assume that $S$ has complexity at least four and that $$\aut\left(\mg{k-1}{S}\right)\simeq \mcg(S).$$
We will show that
$$\aut\left(\mg{k}{S}\right)\simeq \aut\left(\mg{k-1}{S}\right).$$
The idea for the proof of this fact is the same as in Margalit's proof of the rigidity of the pants graph \cite{margalit}. Define the map
\begin{align*}
\theta:E\left(\mg{k}{S}\right)&\longrightarrow V\left(\mg{k-1}{S}\right)\\
e=\alpha\beta&\longmapsto \mu=\alpha\cap \beta.
\end{align*}
It is clearly a surjection.

\begin{lemma}\label{welldef}
Let $A\in\aut(\mg{k}{S})$ and suppose $e,f\in E(\mg{k}{S})$ satisfy $\theta(e)=\theta(f)$. Then $\theta(A(e))=\theta(A(f))$.
\end{lemma}

\begin{proof}
Note that $\theta(e)=\theta(f)=\mu$ if and only if $e,f\in C_\mu$. This implies that $A(e),A(f)\in A(C_\mu)$. The automorphism $A$ induces the map
\begin{align*}
\psi:\cc(S\setminus\mu)&\rightarrow \mg{k}{S}\\
a&\mapsto A(\{a\}\cup\mu)
\end{align*}
which is a simplicial embedding. So, by the results of Section \ref{base}, there exists a $(k-1)$-multicurve $\nu$ such that the image of $\psi$ is contained in $C_\nu$. This implies that $\theta(A(e))=\nu=\theta(A(f)).$
\end{proof}

\begin{rmk}
Since we need only the existence of $\nu$ such that $\im(\psi)\subseteq C_\nu$ and not the full base case, we do not need any condition on the complexities of the components of $S\setminus \mu$.
\end{rmk}

By Lemma \ref{welldef} any $A\in \aut\left(\mg{k}{S}\right)$ induces a well-defined map
$$\varphi(A): V\left(\mg{k-1}{S}\right)\rightarrow V\left(\mg{k-1}{S}\right)$$
given by
$$\varphi(A)(\mu)=\theta(A(e))$$
where $e$ is any edge of $\mg{k}{S}$ such that $\theta(e)=\mu$. Note that $\varphi(A)$ is actually a bijection, since it is easy to show that its inverse exists and is equal to $\varphi(A^{-1})$.

We want to prove that $\varphi(A)$ is an automorphism of $\mg{k-1}{S}$. We first show that it sends edges to edges.

\begin{lemma}
For any $A\in\aut\left(\mg{k}{S}\right)$, $\varphi(A)$ sends edges to edges.
\end{lemma}
\begin{proof}
Suppose $\mu\nu$ is an edge of $\mg{k-1}{S}$. This means that the intersection $C_\mu\cap C_\nu$ is not empty, so it contains exactly one vertex $\mu\cup\nu$. As a consequence $A(C_\mu)$ and $A(C_\nu)$ intersect in exactly one vertex too and thus $C_{\varphi(A)(\mu)}\cap C_{\varphi(A)(\nu)}$ is not empty. By remark \ref{int_complete}, this means that either $C_{\varphi(A)(\mu)}=C_{\varphi(A)(\nu)}$ or $C_{\varphi(A)(\mu)}\cap C_{\varphi(A)(\nu)}$ is one vertex. But if $C_{\varphi(A)(\mu)}=C_{\varphi(A)(\nu)}$, then $\varphi(A)(\mu)=\varphi(A)(\nu)$, which contradicts the injectivity of $\varphi(A)$. Hence $C_{\varphi(A)(\mu)}\cap C_{\varphi(A)(\nu)}$ is one vertex, so $\varphi(A)(\mu)\cup\varphi(A)(\nu)$ is a $k$-multicurve, i.e.\ that $\varphi(A)(\mu)\varphi(A)(\nu)$ is an edge of $\mg{k-1}{S}$.
\end{proof}

We have seen $\varphi(A)$ is a simplicial map of $\mg{k-1}{S}$ to itself and it is a bijection. As the same holds for $\varphi(A)^{-1}=\varphi(A^{-1})$, $\varphi(A)$ is an automorphism of $\mg{k-1}{S}$. Therefore we have a map
$$\varphi:\aut\left(\mg{k}{S}\right)\rightarrow \aut\left(\mg{k-1}{S}\right).$$
We claim that $\varphi$ is a group isomorphism.

\begin{lemma}
The map $\varphi$ is a group homomorphism.
\end{lemma}

\begin{proof}
Consider $A,B\in\aut\left(\mg{k}{S}\right)$ and let $\mu$ be $(k-1)$-multicurve. Suppose $\mu=\theta(e)$. Then
\begin{gather*}
\varphi(AB)(\mu)=\theta(AB(e))=\theta(A(B(e)))=\varphi(A)(\theta(B(e)))=\\
=\varphi(A)(\varphi(B)(\theta(e)))=(\varphi(A)\circ\varphi(B))(\mu),
\end{gather*}
which shows that $\varphi(AB)=\varphi(A)\circ\varphi(B)$, as claimed.
\end{proof}

\begin{lemma}
The map $\varphi$ is injective.
\end{lemma}
\begin{proof}
Since $\varphi$ is a group homomorphism, it is enough to prove that its kernel is reduced to the identity. Suppose that $\varphi(A)$ it the identity on $\mg{k-1}{S}$ and let $\mu=\{\alpha_1\dots,\alpha_k\}$ be a $k$-multicurve. Define $\mu_i$ to be $\mu\setminus\{\alpha_i\}$, for any $i$ between $1$ and $k$. Then $\mu_i=\theta(e_i)$, where we can choose $e_i$ incident to $\alpha$ for every $i$ (say $e_i=\mu\nu_i$). Since by hypothesis $\varphi(A)$ is the identity, we obtain
$$\mu_i=\varphi(A(\mu_i))=\theta(A(e_i))=A(\mu)\cap A(\nu_i)$$
which means that $\mu_i\subseteq A(\mu)$ for every $i$. So $A(\mu)=\mu$, that is, $A$ is the identity in $\mg{k}{S}$. 
\end{proof}

\begin{lemma}
The map $\varphi$ is surjective.
\end{lemma}
\begin{proof}
Let $F:\mcg(S)\rightarrow\aut\left(\mg{k}{S}\right)$ and $G:\mcg(S)\rightarrow\aut\left(\mg{k-1}{S}\right)$ be the natural maps from the mapping class group to the multicurve graphs. By the induction hypothesis we know that $G$ is surjective. To show that $\varphi$ is surjective it then enough to prove that $G=\varphi\circ F$. Let $f$ be any mapping class and $\mu=\{a_1,\dots,a_{k-1}\}$ any $(k-1)$-multicurve. Suppose $e=\alpha\beta$ is an edge of $\mg{k}{S}$ such that $\theta(e)=\mu$; say $\alpha=\mu\cup\{a\}$ and $\beta=\mu\cup \{b\}$. Then
\begin{gather*}
\varphi(F(f))(\mu)=\theta\left(F(f)(\alpha)F(f)(\beta)\right)=F(f)(\alpha)\cap F(f)(\beta)=\\=\{f(a_1),\dots,f(a_{k-1}),f(a)\}\cap\{f(a_1),\dots,f(a_{k-1}),f(b)\}=\\
=\{f(a_1),\dots,f(a_{k-1})\}=G(f)(\mu).
\end{gather*}
\end{proof}
So $\aut\left(\mg{k}{S}\right)\simeq \aut\left(\mg{k-1}{S}\right)\simeq \mcg(S)$.\\

\section{Proof of Theorem \ref{embedding_intro}}\label{induction}
We prove Theorem \ref{embedding_intro} by induction on $k_1$. The base case, for $k_1=1$, is proven in Section \ref{base}.

For the induction step, assume the theorem holds until $k_1-1$ and consider a simplicial embedding $\varphi:\mg{k_1}{S_1}\hookrightarrow \mg{k_2}{S_2}$ satisfying the hypothesis of Theorem \ref{embedding_intro}. If $\alpha$ is a nonseparating curve on $S_1$, $\varphi$ induces a simplicial embedding
$$\varphi_\alpha:\mg{k_1-1}{S_1\setminus \alpha}\simeq C_\alpha\hookrightarrow\mg{k_2}{S_2}$$
which satisfies the hypothesis of Theorem \ref{embedding_intro}. So by induction we know that $$\xi(S_2)-\xi(S_1\setminus \alpha)=k_2-(k_1-1),$$
thus $$\xi(S_2)-\xi(S_1)=k_2-k_1.$$
Define $d:=k_2-k_1$. The induction hypothesis also tells us that $\varphi_\alpha$ is induced by a $\pi_1$-injective embedding $f_\alpha$ of $S_1\setminus\alpha$ into $S_2$ and a $(d+1)$-multicurve $\nu_\alpha$ of $S_2$.

If $\alpha$ is an outer curve on $S_1$ which cuts off a pair of pants $P_\alpha$, the same holds if we replace $S_1\setminus \alpha$ by $S_1\setminus\alpha\setminus P_\alpha$.

For ease of notation, in what follows we interpret a $d$-multicurve when $d=0$ as the empty set. Moreover, given a possibly disconnected surface $S$, we will denote by $\pos{S}$ the union of the components of $S$ which are not pairs of pants. For instance, if $\alpha$ is an outer curve on $S_1$ cutting off a pair of pants $P_\alpha$, then $\pos{S_1\setminus \alpha}=S_1\setminus \alpha\setminus P_\alpha$.

\begin{lemma}\label{nicepair}
If $\alpha$ and $\beta$ form a nice pair, then there exists a $\pi_1$-injective embedding $f:S_1\rightarrow S_2$ and a $d$-multicurve $\nu$ on $S_2$ such that $\varphi$ is induced by $f$ and $\nu$ on  $C_\alpha\cup C_\beta$. Moreover, if $d=0$, $f$ can be taken to be a homeomorphism.
\end{lemma}

Note that if $\alpha, \beta$ form a nice pair, then $\pos{S_1\setminus(\alpha\cup\beta)}$ has complexity $\xi(S_1)-2\geq k_1+2\geq4$ and hence satisfies Theorem \ref{rigidity_intro}. In particular, the curve graph $\cc(\pos{S_1\setminus(\alpha\cup\beta)})$ and, for $k_1\geq3$, the multicurve graph $\mg{k_1-2}{\pos{S_1\setminus(\alpha\cup\beta)}}$ are rigid. 

Assuming Lemma \ref{nicepair}, fix a nice pair $(\alpha,\beta)$ and consider $f$ and $\nu$ obtained via the lemma. We will prove:

\begin{lemma}\label{nonsep}
The embedding $\varphi$ is induced by $f$ and $\nu$ on $\bg{k_1}{S_1}$.
\end{lemma}
We will end the proof showing that if $\varphi$ is induced by $f$ and $\nu$ on $\bg{k_1}{S_1}$, then it is induced by $f$ and $\nu$ everywhere.

\begin{proof}[Proof of Lemma \ref{nicepair}]
As seen at the beginning of this section, by induction we get two $\pi_1$-injective embeddings $f_\alpha:\pos{S_1\setminus \alpha}\rightarrow S_2$, $f_\beta:\pos{S_1\setminus \beta}\rightarrow S_2$ and two $(d+1)$-multicurves on $S_2$ such that $\varphi$ is induced by $f_\alpha, \nu_\alpha$ on $C_\alpha$ and by $f_\beta,\nu_\beta$ on $C_\beta$.

Our objective is to show that $\nu_\alpha\cap\nu_\beta$ is the $d$-multicurve we are looking for (in particular, it is empty when $d=0$) and that $f_\alpha$ and $f_\beta$ define a $\pi_1$-injective embedding of $S_1$ into $S_2$ which is a homeomorphism when $d=0$.

{\bf Case $\bm{k_1=2}$:} since $\varphi(\{\alpha,\beta\})=\nu_\alpha\cup \{f_\alpha(\beta)\}=\nu_\beta\cup \{f_\beta(\alpha)\}$, $\nu_\alpha\cap\nu_\beta$ is a $d$-multicurve if (and only if) $\nu_\alpha\neq\nu_\beta$. So suppose by contradiction that $\nu_\alpha=\nu_\beta$. We know that
$$\varphi\left(C_\alpha\simeq \cc(\pos{S_1\setminus \alpha)}\right)\subseteq C_{\nu_\alpha}\simeq \cc(\pos{S_2\setminus \nu_\alpha})$$
and
$$\varphi\left(C_\beta\simeq \cc(\pos{S_1\setminus \beta)}\right)\subseteq C_{\nu_\alpha}\simeq \cc(\pos{S_2\setminus \nu_\alpha}).$$
But if we compute the complexities we have
$$\xi(\pos{S_2\setminus \nu_\alpha})=\xi(S_2)-(d+1)=(\xi(S_2)-d)-1=\xi(S_1)-1=\xi(\pos{S_1\setminus \alpha})=\xi(\pos{S_1\setminus\beta}).$$
By a result in \cite{shackleton}, the inclusions are actually equalities, i.e.\ $\varphi(C_\alpha)=\varphi(C_\beta)=C_{\nu_\alpha}.$

Choose $\mu\in C_{\nu_\alpha}$ different from $\varphi(\{\alpha,\beta\})$. Then there exist $a\in C_\alpha$ and $b\in C_\beta$ such that $\varphi(a)=\varphi(b)=\mu$. By the injectivity of $\varphi$ this means that $a=b\in C_\alpha\cap C_\beta$. But the only multicurve in the intersection is $\{\alpha,\beta\}$ and this contradicts the fact that $\mu\neq \varphi(\{\alpha,\beta\})$. So $\nu_\alpha$ and $\nu_\beta$ must be different and $\nu:=\nu_\alpha\cap\nu_\beta$ is a $d$-multicurve. In particular, if $d=0$, $\nu_\alpha$ and $\nu_\beta$ are distinct single curves and $\nu$ is empty.

Consider any curve $\gamma$ on $S_1\setminus(\alpha\cup\beta)$. As $\{\alpha,\gamma\}$ and $\{\beta,\gamma\}$ are multicurves joined by an edge, so are their images via $\varphi$. We have
\begin{gather*}\varphi(\{\alpha,\gamma\})=\nu_\alpha\cup \{f_\alpha(\gamma)\}\\
\varphi(\{\beta,\gamma\})=\nu_\beta\cup \{f_\beta(\gamma)\}
\end{gather*}
and since $\nu_\alpha\neq \nu_\beta$, we must have $f_\alpha(\gamma)=f_\beta(\gamma)$. Thus $f_\alpha$ and $f_\beta$ induce the same map
$$\cc(\pos{S_1\setminus (\alpha\cup\beta)})\rightarrow \cc(\pos{S_2\setminus\varphi(\{\alpha,\beta\})}).$$
As $\xi(\pos{S_1\setminus (\alpha\cup\beta)})=\xi(\pos{S_2\setminus\varphi(\{\alpha,\beta\})})$, we can assume that $g_\alpha=f_\alpha\big|_{\pos{S_1\setminus (\alpha\cup\beta)}}$ and $g_\beta=f_\beta\big|_{\pos{S_1\setminus (\alpha\cup\beta)}}$ are homeomorphisms onto $\pos{S_2\setminus\varphi(\{\alpha,\beta\})}$. Moreover $g_\beta^{-1}\circ g_\alpha$ induces the identity map on $\cc(\pos{S_1\setminus (\alpha\cup\beta)})$, which implies (by the rigidity of the curve graph) that the class of $g_\beta^{-1}\circ g_\alpha$ is trivial in $\mcg{\left(\pos{S_1\setminus(\alpha\cup\beta)}\right)}$. Thus $g_\alpha=g_\beta\circ h$, where $h$ is a diffeomorphism of $\pos{S_1\setminus(\alpha\cup\beta)}$ isotopic to the identity. Up to replacing $f_\beta$ with an extension of $g_\beta\circ h$, we can assume $f_\alpha$ and $f_\beta$ agree on $\pos{S_1\setminus (\alpha\cup\beta)}$. Hence they define a map $f:S_1\rightarrow S_2$. By elementary topology arguments, $f$ is a $\pi_1$-injective embedding.

When $d=0$, $f(\alpha)=\nu_\alpha$ and $f(\beta)=\nu_\beta$. In particular, we know that $\nu_\alpha$ and $\nu_\beta$ are of the same type (outer or nonseparating) as $\alpha$ and $\beta$ (by \cite[Lemmas 10 and 11]{shackleton}). Now $f_\alpha$ is a map inducing a simplicial embedding between the curve complexes of $\pos{S_1\setminus \alpha}$ and $\pos{S_2\setminus \nu_\alpha}$. The surfaces have the same complexity, so by induction\footnote{Actually, for $k_1=2$ the tool we need is Shackleton's result in \cite{shackleton} which is in fact what we use in the base case. For $k_1\geq 3$ we really need the induction hypothesis.} we could have chosen it to be a homeomorphism between $\pos{S_1\setminus\alpha}$ and $\pos{S_2\setminus\nu_\alpha}$. The same holds for $f_\beta$. As a consequence $f$ can be chosen to be a homeomorphism between $S_1$ and $S_2$.

{\bf Case $\bm{k_1\geq 3}$:} consider $\mu\in C_{\alpha\cup\beta}$, that is, $\mu=\{\alpha\}\cup\{\beta\}\cup \widetilde{\mu}$, where $\widetilde{\mu}$ is a $(k_1-2)$-multicurve. Then
\begin{align*}
\varphi(\mu)&=\nu_\alpha\cup f_\alpha(\mu\setminus\{\alpha\})=\nu_\alpha\cup \{f_\alpha(\beta)\}\cup f_\alpha(\widetilde{\mu})\\
&=\nu_\beta\cup f_\beta(\mu\setminus\{\beta\})=\nu_\beta\cup \{f_\beta(\alpha)\}\cup f_\beta(\widetilde{\mu}).
\end{align*}
If we let $\mu$ vary in $C_{\alpha\cup\beta}$, i.e.\ we let $\widetilde{\mu}$ vary in $\mg{k_1-2}{S_1\setminus(\alpha\cup\beta)}$, by the injectivity of $\varphi$ $f_\alpha(\widetilde{\mu})$ and $f_\beta(\widetilde{\mu})$ vary (and the rest is fixed, since it doesn't depend on $\widetilde{\mu}$), so
$$\nu_\alpha\cup \{f_\alpha(\beta)\}=\nu_\beta\cup \{f_\beta(\alpha)\}\;\mbox{  and  }\; f_\alpha(\widetilde{\mu})=f_\beta(\widetilde{\mu}).$$
This implies that $f_\alpha$ and $f_\beta$ induce the same map from
$$\mg{k_1-2}{\pos{S_1\setminus(\alpha\cup\beta)}}\rightarrow\mg{k_1-2}{\pos{S_2\setminus(\nu_\alpha\cup f_\alpha(\beta))}}$$
and with the same argument as in the case $k_1=2$, this time by Theorem \ref{rigidity_intro}, we obtain that $f_\alpha\big|_{\pos{S_1\setminus(\alpha\cup\beta)}}=f_\beta\big|_{\pos{S_1\setminus(\alpha\cup\beta)}}\circ h$, where $h$ is a diffeomorphism of $S_1\setminus(\alpha\cup\beta)$ isotopic to the identity. Again, up to replacing $f_\beta$ with an extension of $f_\beta\big|_{S_1\setminus(\alpha\cup\beta)}\circ h$, we can assume that $f_\alpha$ and $f_\beta$ agree on $\pos{S_1\setminus (\alpha\cup\beta)}$, so that they define a map $f:S_1\rightarrow S_2$. As in the case $k_1=2$, $f$ is a $\pi_1$-injective embedding and a homeomorphism if $d=0$. As $f$ is injective, $f(\alpha)\neq f(\beta)$. This implies that $\nu_\alpha\cap\nu_\beta$ is a $d$-multicurve, which we denote by $\nu$.
\end{proof}

\begin{rmk} Different choices in the construction above give different maps from $S_1$ to $S_2$. This is not surprising, since if there is a map inducing a simplicial embedding, it is not unique. For instance, if $\psi:\mg{k_1}{S_1}\hookrightarrow \mg{k_2}{S_2}$ is a simplicial embedding induced by a map $g$, it is also induced by $g\circ h$, where $h$ is any homeomorphism of $S_1$ isotopic to the identity. On the other hand, the multicurve is uniquely determined by the simplicial embedding.
\end{rmk}

\begin{proof}[Proof of Lemma \ref{nonsep}]
Let $f$ and $\nu$ be given by the nice pair $(\alpha, \beta)$ according to Lemma \ref{nicepair}. Consider any curve $\gamma\in\bc(S_1)$. Since $\bc(S_1)$ is connected, there is a path
$$\gamma_0=\alpha,\gamma_1,\dots,\gamma_m=\gamma.$$
Consider $\alpha$ and $\gamma_1$. As they form a nice pair, they determine a map $f'$ and a $d$-multicurve $\nu'$ by Lemma \ref{nicepair} and we can assume $f'=f$ on $\pos{S_1\setminus \alpha}$ (as they both agree with $f_\alpha$). Note that this implies $\nu'=\nu$ since $f'(\mu)\cup\nu'=\varphi(\mu)=f(\mu)\cup\nu$ for any $k$-multicurve $\mu$ on $\pos{S_1\setminus \alpha}$. Our goal is to show that $f'=f$ on $S_1$. This implies, by repeating the argument along the path, that $\varphi$ is induced by $f$ and $\nu$ on $C_{\gamma}$.

If $\alpha$ is nonseparating then we must have $f=f'$ on $S_1$ by continuity. So assume $\alpha$ is an outer curve. Since $f'=f$ on $\pos{S_1\setminus \alpha}$, the two maps differ a priori by a twist about $\alpha$. To see that this is not the case we will show there exists a curve $\delta$ intersecting $\alpha$ such that $f(\delta)=f'(\delta)$. 

First consider the case when $\beta$ and $\gamma_1$ are disjoint. Let $\delta$ be a curve disjoint from both $\beta$ and $\gamma_1$ but intersecting $\alpha$. If $k_1\geq3$ there exists a $k_1$-multicurve $\mu$ containing $\beta, \gamma_1$ and $\delta$. Since $\mu\in C_{\beta}\cup C_{\gamma_1}$, 
$$f(\mu)\cup\nu=\varphi(\mu)=f'(\mu)\cup\nu.$$
Moreover, $\mu\setminus\delta\in \pos{S_1\setminus \alpha}$, so $f(\mu\setminus\delta)=f'(\mu\setminus\delta)$ and we must have $f(\delta)=f'(\delta)$. If $k_1=2$ we instead consider the multicurves $\mu_1=\{\beta, \delta\}$ and $\mu_2=\{\gamma_1, \delta\}$. As $\mu_1$ and $\mu_2$ are adjacent in the multicurve graph, so are $\varphi(\mu_1)$ and $\varphi(\mu_2)$. In particular they differ by exactly one curve. Now, 
$$\varphi(\mu_1)=f(\beta)\cup f(\delta)\cup\nu$$
and
$$\varphi(\mu_2)=f'(\gamma_1)\cup f'(\delta)\cup\nu=f(\gamma_1)\cup f'(\delta)\cup\nu.$$
This implies $f(\delta)=f'(\delta)$ because $f(\beta)\neq f(\gamma_1)$. 

Now consider the case when $\beta$ and $\gamma_1$ intersect. One can show that there is a path $$\beta_0=\beta,\beta_1,\dots,\beta_m=\gamma_1$$ in $\cc{(\pos{S_1\setminus \alpha})}$ such that $(\alpha,\beta_i)$ form a nice pair for each $i$. To prove this, the idea is to construct curves in $\bc(\pos{S_1\setminus\alpha})$ with a similar procedure as the one for Lemma \ref{bconn}. One just needs to be careful not to choose in the process any outer curve of $\pos{S_1\setminus \alpha}$ which bounds a pair of pants with $\alpha$ on $S_1$. Such a choice would give a curve which is not outer in $S_1$ (as in Figure \ref{nonouter}). This can be done because $\xi(\pos{S_1\setminus \alpha})=\xi(S_1)-1\geq 5$.

\begin{figure}[h]
\begin{overpic}{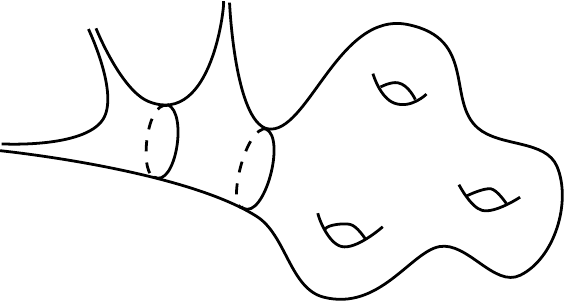}
\put (24,13) {$\alpha$}
\put (40,6) {$\delta$}
\end{overpic}
\caption{$\delta$ is an outer curve in $\pos{S_1\setminus \alpha}$, but not in $S_1$.}\label{nonouter}
\end{figure}

By replacing $\beta$ and $\gamma_1$ by $\beta_i$ and $\beta_{i+1}$ we can apply the above argument along this path to obtain again that $f(\delta)=f'(\delta)$. 
\end{proof}

The final step is showing that if $\varphi$ is induced by $f$ and $\nu$ on $\bg{k_1}{S_1}$, then it is is induced by $f$ and $\nu$ everywhere.

\begin{lemma}\label{image}
The image of $\varphi$ is in $C_\nu$.
\end{lemma}
\begin{proof}
Suppose not. Then let $\mu\in\mg{k_1}{S_1}\setminus\bg{k_1}{S_1}$ such that $\varphi(\mu)\notin C_\nu$, i.e.\ that there is some curve $\alpha\in \nu$ such that $\alpha\notin\varphi(\mu)$. Consider any $\bar{\mu}\in \bg{k_1}{S_1}$ adjacent to $\mu$ ($\bar{\mu}$ exists by Remark \ref{adjacent}). Its image is
$$\varphi(\bar{\mu})=\nu\cup f(\bar{\mu}).$$
Then
$$|\varphi(\mu)\cap\varphi(\bar{\mu})|=k_2-1,$$
so
$$\varphi(\mu)\cap\varphi(\bar{\mu})=\nu\setminus \{\alpha\}\cup f(\bar{\mu}).$$
This implies that
$$f(\bar{\mu})\subseteq \varphi(\mu)\setminus \nu=: \eta.$$
Note that as $\varphi(\mu)$ and $\varphi(\bar{\mu})$ are adjacent, they differ by one curve, which must be $\alpha$. Thus $\nu\setminus\alpha\subseteq \varphi(\mu)$ and $\eta$ has cardinality $k_2-(d-1)=k_1+1.$ So $\bar{\mu}$ is contained in $f^{-1}(\eta)$, which has cardinality at most $k_1+1$ since $f$ is injective. 

This holds for every  $\bar{\mu}\in \bg{k_1}{S_1}$ that is adjacent to $\mu$, and $\eta$ does not depend on $\bar{\mu}$. There are at most $k_1+1$ multicurves contained in a set of cardinality $k_1+1$, but we know from Remark \ref{adjacent} that $\mu$ has infinitely many adjacent multicurves in $\bg{k_1}{S_1}$, a contradiction.
\end{proof}

As a consequence, we obtain a map $\Phi:\mg{k_1}{S_1}\rightarrow\mg{k_1}{S_2}$ given by 
$$\Phi(\mu):= \varphi(\mu)\setminus \nu.$$
Clearly
$$\varphi(\mu)=\Phi(\mu)\cup \nu$$
for every $\mu\in\mg{k_1}{S_1}$. So we just want to show that $\Phi$ is (induced by) $f$.
\begin{lemma}
For any $\beta$ on $S_1$, $\Phi(C_{\beta})\subseteq C_{f(\beta)}$.
\end{lemma}
\begin{proof}
The idea of the proof is the same as for Lemma \ref{image}.

If $\mu\in C_{\beta}\cap \bg{k_1}{S_1}$, then $\Phi(\mu)=f(\mu)\in C_{f(\beta)}$.

Suppose now $\mu\in C_{\beta}$ is not in $\bg{k_1}{S_1}$. Any $\bar{\mu}\in \bg{k_1}{S_1}$ adjacent to $\mu$ satisfies
$$\Phi(\bar{\mu})\cap \Phi(\mu)=f(\bar{\mu})\setminus \{f(\beta)\}$$
so $\bar{\mu}$ is a subset $\Phi(\mu)\cup \{f(\beta)\}$. This implies that there are finitely many multicurves in $\bg{k_1}{S_1}$ adjacent to $\mu$, contradicting Remark \ref{adjacent}.
\end{proof}

Now take any $\mu\in \mg{k_1}{S_1}$, say $\mu=\{\beta_1,\dots,\beta_{k_1}\}$. Then
$$\{\Phi(\mu)\}=\bigcap_{j=1}^{k_1} \Phi(C_{\beta_j})\subseteq \bigcap_{j=1}^{k_1} C_{f(\beta_j)}={f(\mu)}.$$
So $\Phi(\mu)=f(\mu)$ for all $\mu\in\mg{k_1}{S_1}$, which ends the proof of the theorem.
\bibliographystyle{alpha}
\bibliography{references}
\end{document}